\documentclass[preprint,3p]{elsarticle}

\usepackage{amssymb}

\usepackage{amsmath}
\usepackage{amsthm}
\usepackage[usenames]{color}
\usepackage{eepic,epic}

\newtheorem{thm}{Theorem}[section]

\newtheorem{lem}[thm]{Lemma}

\theoremstyle{definition}

\theoremstyle{remark}

\numberwithin{equation}{section}

\begin{document}

\begin{frontmatter}



\title{Asymptotic profile of solutions to the heat equation on thin plate with boundary heating}

\author[SKKU1]{Eun-ho Lee}
\ead{e.h.lee@skku.edu}

\author[SKKU2]{Woocheol Choi\corref{Corr2}}
\ead{choiwc@skku.edu}
\cortext[Corr2]{Corresponding author.}
\address[SKKU1]{School of Mechanical Engineering, Sungkyunkwan University, Suwon, Gyeonggi-do, Republic of Korea}
\address[SKKU2]{Department of Mathematics, Sungkyunkwan University, Suwon, Gyeonggi-do, Republic of Korea}

\begin{abstract}
In this section, we consider the heat equation on a plate with thickness $h>0$ being heated by a heat source on upper and lower faces of the plate. We obtain an asymptotic profile of the solution as the thickness $h>0$ approaches to zero.
\end{abstract}


\end{frontmatter}


\section{Introduction}

Industries have used focused surface heating for engineering purposes. Food industry has employed infrared (IR) energy to heat surface of foods for blanching, drying, roasting, and thawing processes \cite{R,S}. Manufacturing industries also have done the IR surface heating for metal forming \cite{KLB, LYY2} and soldering \cite{AFM, L} processes. Many of these industrial applications have based on trial and error experiences, study of mathematical analysis has not been sufficiently reported. A few studies of the focused IR heating analysis are about design of the reflectors in IR heating \cite{LYY, LY}, multiphysics simulation \cite{LK}, and friction stir welding {HSBS}. Speciﬁcally, this problem is modelled by the following heat equation with Robin boundary condition
\begin{equation}\label{eq-1-1}
\begin{array}{rll}
\partial_t U(\mathbf{x},t) - \Delta U (\mathbf{x},t) &=0&~ \textrm{in}~\Omega_h \times [0,\infty)
\\
U(\mathbf{x},0)&=G(\mathbf{x})&~\textrm{on}~\Omega_h
\\
\frac{\partial U}{\partial x_3}(\mathbf{x},t) &=F(\mathbf{x},t)+a[T_0- U(\mathbf{x},t)] &~\textrm{on}~P \times \{h\} \times [0,\infty)
\\
\frac{\partial U}{\partial x_3}(\mathbf{x},t) & = -F(\mathbf{x},t) - a[T_1 - U(\mathbf{x},t)]&~\textrm{on}~P \times \{0\} \times [0,\infty)
\\
\frac{\partial U}{\partial \nu} (\mathbf{x},t)&=0& ~\textrm{on}~\partial P \times [0,h] \times [0,\infty),
\end{array}
\end{equation}
where $\Omega_{h} = P \times [0,h]$ is a plate with small thickness $h>0$ and  $P \subset \mathbb{R}^2$. In this paper, we are interested in asymptotic profile of solution $U(\mathbf{x},t)$ to \eqref{eq-1-1} as the thickness $h>0$ of the plate approaches to zero. In \eqref{eq-1-1} no heat change is assumed on $\partial P \times [0,h]$ as the surface measure of $\partial P \times [0,h]$ is much smaller than that of the upper and lower boundary $P \times \{0,h\}$. 

In the literature, there have been a lot of interest in studying a similar problem to \eqref{eq-1-1}, namely the reaction diffusion equation on thin domains with Neumann boundary condition
\begin{equation}\label{eq-1-51}
\begin{array}{rll}
\partial_t U(\mathbf{x},t) - \Delta U(\mathbf{x},t) &= f(U) (\mathbf{x},t)  &\textrm{in}~\Omega_h \times [0,\infty)
\\
U(\mathbf{x},0)&=G(\mathbf{x})&\textrm{on}~ \Omega_h 
\\
\frac{\partial U}{\partial \nu}(\mathbf{x},t) &=0& \textrm{on}~\partial \Omega_h \times [0,\infty).
\end{array}
\end{equation}
The mathematical analysis for the asymptotic limit $h\rightarrow 0$ of \eqref{eq-1-51} was initiated by Hale and Raugel \cite{HR1} where the authors raised a general question: \emph{If we consider an evolution equation on a spatial domain $\Omega$ such that $\Omega$ is small in a direction, is it possible to approximate the dynamics by an equation on a lower dimensional spatial domain?} 
This question was answered affirmatively for problem \eqref{eq-1-51} in \cite{HR1} and extended to various settings \cite{HR2, LLW, LWW, C, AP}. More precisely, Hale and Raguel \cite{HR1} showed that the solution $U(\mathbf{x},t)$ to \eqref{eq-1-51} is approximated by the two dimensional problem 
\begin{equation}\label{eq-1-52}
\begin{array}{rll}
\partial_t u(x,t) - \Delta u(x,t) &= f(u) (x,t)  &\textrm{in}~P \times [0,\infty)
\\
u(x,0)&=g(x)&\textrm{on}~P 
\\
\frac{\partial u}{\partial \nu}(x,t) &=0& \textrm{on}~\partial P  \times [0,\infty)
\end{array}
\end{equation}
as the thickness $h>0$ get close to zero, where $g: P \rightarrow \mathbb{R}$ is properly chosen in terms of $G$.

The problem \eqref{eq-1-1} models a thin plate with an heat source $F$ on boundary and the main concern is the effect of the heat source $F$ on temperature of plate when $h>0$ is very small. Therefore it is admissible to consider the case that the initial state $G$ is static, i.e., 
\begin{equation*}
\begin{array}{rll}
 - \Delta G(\mathbf{x}) &=0&~ \textrm{in}~\Omega_h 
\\
\frac{\partial G}{\partial x_3}(\mathbf{x}) &=a[T_0 - G(\mathbf{x})] &~\textrm{on}~P \times \{h\}
\\
\frac{\partial G}{\partial x_3}(\mathbf{x}) & = - a[T_1 - G(\mathbf{x})]&~\textrm{on}~P \times \{0\}
\\
\frac{\partial G}{\partial \nu}(\mathbf{x}) &=0& ~\textrm{on}~\partial P \times [0,h].
\end{array}
\end{equation*}
As far as we know, there has been no results on the asymptotic profile for the heat equation with the Robin boundary condition on thin plate.  We observe in problem  \eqref{eq-1-1} that the termperature on plate interacts with the outside termperature since the convection  coefficient $a>0$ is nonzero. Therefore the asymptotic behavior as $h\rightarrow 0$ should be different from the case $a=0$ because the interaction could effect more the temperature inside of the plate if the thickness of  plate is more thin. Now we state the main result of this paper. 
\begin{thm}\label{thm-1}Let $U \in C^2 ([0,\infty); \Omega_h)$ be a solution to \eqref{eq-1-1} with $F \in L^{\infty} ([0,\infty); \Omega_h)$. We assume that $h \in (0, 1/3a)$ and let $\alpha_1=\alpha_1 (h)$ be the smallest positive solution of
\begin{equation}
\tan (h q ) =  \frac{2 a q}{q^2 - a^2}.
\end{equation}
Then for each $(x,x_3) \in P \times [0,h]$ and $t\geq 0$, the solution $U$ satisfies
\begin{equation*}
\begin{split}
&\left|U((x,x_3),t) -\left( G(x,x_3) +  \frac{\alpha_1^2}{2a}\int_0^{t} \int_{P \times \{0,h\}} e^{-\alpha_1^2 (t-s)} W(x,t,y,s) F(y,y_3, s) dS_y ds\right) \right|
\\
&\quad \leq \frac{19 h}{3} \|F\|_{L^{\infty}([0,t])}
\end{split}
\end{equation*}
 where $W(x,t,y,s)$ denotes the Green's function of the heat equation on the two dimensional domain $P$ with Neumann boundary  condition
\begin{equation*}
\begin{array}{rll}
 \partial_t u(x,t)- \Delta u(x,t) &=0&~ \textrm{in}~P \times [0,\infty)
\\
u(x,0)&=g(x)&~\textrm{on}~P
\\
\frac{\partial u(x,t)}{\partial \nu} &=0& ~\textrm{on}~\partial P \times [0,\infty).
\end{array}
\end{equation*}
Here we denoted by $\int_{P \times \{0,h\}} f(y,y_3) dS_y$ the sum $\int_{P} f(y,0) dy + \int_{P} f(y,h) dy$ for integrable function $f: P \times\{0,h\} \rightarrow \mathbb{R}$ and $\|F\|_{L^{\infty}([0,t])}:= \sup_{(y,y_3, s) \in P \times \{0,h\} \times [0,t]} |F(y,y_3, s)|$.
\end{thm}
For the value $\alpha_1 = \alpha_1 (h)$ defined in Theorem \ref{thm-1}, we will prove that it satisfies $\lim_{h\rightarrow 0} \frac{\alpha_1 (h)}{\sqrt{2a/h}} =1$ (see  Lemma \ref{lem-3-3}). Based on this property and the integral represntation of Theorem \ref{thm-1}, we will investigate the effect of the thickness $h>0$ and the material property of the plate on the focused IR heating.

In order to prove Theorem \ref{thm-1}, we consider the Green's function $K_h$ associated to \eqref{eq-1-1} and write the solution  $U(\mathbf{x},t)$ to \eqref{eq-1-1} in terms of $K_h$.  The Green's function $K_h$ is known to admits a series expansion of which term depends on $h>0$ and the proof of of Theorem \ref{thm-1} is reduced to study the asymptotic profile of the series expansion of $K_h$ when $h>0$ is small.  We will show that the first term of the series is dominant and the contribution of the other terms can be estimates as $O(h)$.


This paper is organized as follows. In Section 2, we recall the Green's function associated to \eqref{eq-1-1} and study its series expansion. In Section 3 we obtain estimates on the terms in the expansion. This will enable us to prove Theorem \ref{thm-1}.

\section{Green's formula}

In this section, we study the Green's function $K_h(\mathbf{x},t,\mathbf{y},s) :\Pi_h \rightarrow \mathbb{R}^{+}$ associated to \eqref{eq-1-1} defined on 
\begin{equation*}
\Pi_h = \{ (\mathbf{x},t,\mathbf{y},s) \in (\Omega_h \times [0,\infty) )^2 ~:~ t>s\}
\end{equation*}
satisfying for each $(\mathbf{x},t) \in \Omega_h \times [0,\infty)$ that
\begin{equation*}
\begin{array}{rll}
\partial_s K_h (\mathbf{x},t,\mathbf{y},s) - \Delta_y K_h (\mathbf{x},t,\mathbf{y},s) &=0& ~\textrm{in}~\Omega_h \times [0,T]
\\
\frac{\partial K_h}{\partial n_y} (\mathbf{x},t,\mathbf{y},s) + a K_h (\mathbf{x},t,\mathbf{y},s) &=0&~\textrm{on}~(y,s) \in P \times \{0,h\}
\\
\frac{\partial K_h}{\partial n_y} (\mathbf{x},t,\mathbf{y},s) &=0& ~\textrm{on}~\partial P \times [0,h] \times [0,T]
\end{array}
\end{equation*}
and that 
\begin{equation*}
\lim_{t \rightarrow s} K_h (\mathbf{x},t,\mathbf{y},s) = \delta_\mathbf{x} (\mathbf{y}).
\end{equation*}
The solution $U(\mathbf{x},t)$ to \eqref{eq-1-1} is then written as
\begin{equation*}
\begin{split}
U(\mathbf{x},t) =& \int_{\Omega_h} K_h (\mathbf{x},t,\mathbf{y},0) G(\mathbf{y}) d\mathbf{y} 
\\
&\quad+ \int_0^t \int_{\partial P \times \{h\}}  K_h (\mathbf{x},t,\mathbf{y},s) [F(\mathbf{y},s)+a T_0] dS_y ds
\\
&\quad + \int_0^t \int_{\partial P \times \{0\}}  K_h (\mathbf{x},t,\mathbf{y},s) [F(\mathbf{y},s)+a T_1] dS_y ds.
\end{split}
\end{equation*}
Since $G(\mathbf{x})$ is a static state of \eqref{eq-1-1}, we have
\begin{equation}\label{eq-1-2}
\begin{split}
U(\mathbf{x},t) =& G(\mathbf{x}) + \int_0^t \int_{\partial P \times \{0,h\}}  K_h (\mathbf{x},t,\mathbf{y},s) [F(\mathbf{y},s)] dS_y ds.
\end{split}
\end{equation}
As the domain $\Omega_h$ equals to the product $P \times [0,h]$, the Green's function $K_h$ is also a product of two Green's functions corresponding to $P$ and $[0,h]$ described as follows. 

Let $\Psi =\{ (x,t,y,s) \in (P \times [0,T])^2~:~ t>s\}$ and $W: \Psi \rightarrow \mathbb{R}^{+}$ be the Green's function to the problem 
\begin{equation}\label{eq-1-75}
\begin{array}{rll}
\partial_t u(x,t) - \Delta u (x,t) &=0&~ (x,t) \in P \times (0,T)
\\
u(x,0)& =g(x)&~x \in P
\\
\frac{\partial u}{\partial n} (x,t)&=0& ~\textrm{on}~\partial P \times [0,T].
\end{array}
\end{equation}
The existence of the Green's function for the above problem was proved in \cite{CK} for any smooth domain $P \subset \mathbb{R}^2$. Next we consider $\Phi_h := \{ (z,t,w,s) \in ([0,h]\times [0,\infty))^2 : t>s\}$ and the Green's function $G_h : \Phi_h \rightarrow \mathbb{R}^{+}$ to the problem 
\begin{equation*}
\begin{array}{rll}
\partial_t u(z,t) - \Delta u (z,t)&=0&~(z,t) \in [0,h] \times [0,T]
\\
u(z,0)&=g(z)&~z \in [0,h]
\\
\frac{\partial u}{\partial \nu}(z,t) + au(z,t)& =f(z,t)&~(z,t) \in \{0,h\}\times [0,T].
\end{array}
\end{equation*}
The explicit formula of $G_h$ was obtained in \cite{BCHL} as in \eqref{eq-1-4}. 
Now we can state the product formula of $K_h$ appeared in \cite{HSBS}:
\begin{equation}\label{eq-1-3}
K_h (x,x_3, t, y, y_3, s) = W(x,t,y,s) G_h (x_3, t, y_3, s).
\end{equation}
Here $x,y \in \mathbb{R}^2$ and $x_3, y_3 \in \mathbb{R}$.
In order to study the asymptotic behavior of $U(\mathbf{x},t)$ with the formula \eqref{eq-1-2}, we shall investigate the asymptotic behavior of $K_h$. In view of \eqref{eq-1-3}, it is reduced to study the behavior of $G_h$ for small $h>0$. By using the formula of $G_h$ obtained  in \cite{BCHL} we have the following lemma.
\begin{lem}\label{lem-2-1}We have
\begin{equation}\label{eq-1-4}
G_h (z,t,w,s) =\sum_{m=1}^{\infty}P_m (z,t,w,s),
\end{equation}
where for each $m \in \mathbb{N}$,
\begin{equation}\label{eq-1-17}
P_m (z,t,w,s)=\frac{2e^{-\alpha_m^2 (t-s)}[ \alpha_m \cos (\alpha_m z) + a \sin (\alpha_m z)][\alpha_m \cos (\alpha_m w) + a \sin (\alpha_m w)]}{2a + h (a^2 + \alpha_m^2)},
\end{equation}
and  $\alpha_m$ is the $m$-th positive solution $q>0$ of equation
\begin{equation}\label{eq-1-19}
\tan (h q ) =  \frac{2 a q}{q^2 - a^2},
\end{equation}
arranged in increasing order. 
\end{lem}
\begin{proof}
We recall from \cite[605 page]{BCHL} the formula of $G_h$ given as
\begin{equation}\label{eq-1-4}
G_h (z,t,w,s) =\sum_{m=1}^{\infty}P_m (z,t,w,s),
\end{equation}
where
\begin{equation}\label{eq-1-53}
\begin{split}
P_m (z,t,w,s) =&\frac{2}{h} e^{-\beta_m^2 (t-s)/h^2}[ \beta_m \cos (\beta_m z/h) + B \sin (\beta_m z/h)] \\
&\times \frac{[\beta_m \cos (\beta_m w/h) + B \sin (\beta_m w/h)]}{(\beta_m^2 + B^2) [ 1+ B/(\beta_m^2 + B^2)]+B}.
\end{split}
\end{equation}
Here $\beta_m$ are positive solutions to 
\begin{equation}\label{eq-1-16}
\tan \beta_m = \frac{2 \beta_m B}{\beta_m^2 - B^2}\quad \textrm{with}~ B= ah,
\end{equation}
arranged in increasing order for $m \in \mathbb{N}$. Letting $\alpha_m = \beta_m /h$ in \eqref{eq-1-53}, we find that
\begin{equation*}
\begin{split}
&P_m (z,t,w,s)
\\
& = \frac{2h e^{-\alpha_m^2 (t-s)}[ \alpha_m \cos (\alpha_m z) + a \sin (\alpha_m z)][\alpha_m \cos (\alpha_m w) + a \sin (\alpha_m w)]}{2ah + h^2 (a^2 + \alpha_m^2)}
\\
&=\frac{2e^{-\alpha_m^2 (t-s)}[ \alpha_m \cos (\alpha_m z) + a \sin (\alpha_m z)][\alpha_m \cos (\alpha_m w) + a \sin (\alpha_m w)]}{2a + h (a^2 + \alpha_m^2)}. 
\end{split}
\end{equation*}
Also equation \eqref{eq-1-16} is written as
\begin{equation*}
\tan (h \alpha_m ) = \frac{2h^2 a \alpha_m}{h^2 \alpha_m^2 - h^2 a^2} =  \frac{2 a \alpha_m}{\alpha_m^2 - a^2}.
\end{equation*}
The proof is finished.
\end{proof}

\section{Estimates for the asymptotic formula}
In this section, we obtain the estimates for the terms in the expansion of $G_h$ given by \eqref{eq-1-4}. First we shall show that the effect of $P_m$ in \eqref{eq-1-4} with $m \geq 2$ to the solution $U$ of \eqref{eq-1-1} are relatively very small when $h>0$ is close to zero. For this aim we begin with the following lemma. 
\begin{lem}\label{lem-1-1} For $m \geq 2$ we have
\begin{equation*}
\int_0^{t} \int_{P \times \{0,h\}} W(x,t,y,s)P_m (x_3, t,y_3,s) F(y,y_3, s) dS_y ds \leq \frac{8}{h \alpha_m^2}\|F\|_{L^{\infty}( [0,t])}.
\end{equation*}
\end{lem}
\begin{proof}
We estimate \eqref{eq-1-17} as follows
\begin{equation}\label{eq-1-14}
\begin{split}
P_m (x_3,t,y_3,s)& \leq \frac{2 e^{-\alpha_m^2 (t-s)} (\alpha_m + a)^2}{2a + h (a^2 + \alpha_m^2)}
\\
&\leq \frac{4 e^{-\alpha_m^2 (t-s)}(\alpha_m^2 + a^2)}{h (a^2 +\alpha_m^2)} = \frac{4}{h} e^{-\alpha_m^2 (t-s)}.
\end{split}
\end{equation}
Using this we obtain
\begin{equation}\label{eq-1-15}
\begin{split}
A&:=\int_0^{t} \int_{P \times \{0,h\}} W(x,t,y,s)P_m (x_3, t,y_3,s) F(y,y_3, s) dS_y ds 
\\
&~\leq\|F\|_{L^{\infty}( [0,t])}\int_0^{t} \int_{P \times \{0,h\}} W(x,t,y,s) \frac{4}{h} e^{-\alpha_m^2 (t-s)} dS_y ds
\end{split}
\end{equation}
In view of the fact that taking $g\equiv 1$ in \eqref{eq-1-75} implies $u(x,t)\equiv 1$, one has
\begin{equation}\label{eq-1-76}
\int_{P} W(x,t,y,s) dy =1.
\end{equation}
Using this in \eqref{eq-1-15} we obtain
\begin{equation*}
\begin{split}
A&\leq 2 \|F\|_{L^{\infty}( [0,t])} \int_0^{t}  \frac{4}{h} e^{-\alpha_m^2 s} ds
\\
& \leq 2\|F\|_{L^{\infty}( [0,t])} \int_0^{t} \frac{4}{h} e^{-\alpha_m^2 s} ds  =  \frac{8 \|F\|_{L^{\infty}( [0,t])}}{h\alpha_m^2} \int_0^{\alpha_m^2 t} e^{-s} ds
\\
& \leq \frac{8\|F\|_{L^{\infty}( [0,t])}}{h \alpha_m^2}.
\end{split}
\end{equation*}
The proof is finished.
\end{proof}
Next we find the following estimates on $\alpha_m$ for $m \geq 2$. 
\begin{lem}\label{lem-3-2}Assume that $h < \frac{\pi}{2a}$. For $m \geq 2$, we have $\alpha_m \in \left[ \frac{(m-1)\pi}{h}, \frac{(m-1)\pi}{h} + \frac{\pi}{2h}\right]$.
\end{lem}
\begin{proof}
For $q \geq 0$ we let $\Phi(q) = \frac{2a q}{q^2 -a^2}$. From \eqref{eq-1-19} we see that $\alpha_m$ is the $m$-th positive solution of 
\begin{equation*}
\tan (hq) = \Phi(q).
\end{equation*}
Using an elementary calculus, we find that
\begin{itemize}
\item For $q \in (0,a)$, the function $\Phi(q)$ is negative and decreasing function with 
\begin{equation}\label{eq-1-11}
\Phi(0) = 0\quad \textrm{and}\quad \lim_{q \rightarrow a^{-}}\Phi(q) = -\infty.
\end{equation}
\item For $q \in (a, \infty)$, the function $\Phi(q)$ is positive and decreasing function with 
\begin{equation}\label{eq-1-12}
\lim_{q \rightarrow a^{+}} \Phi(q) = \infty \quad \textrm{and}\quad\lim_{x \rightarrow \infty} \Phi(q) = 0.
\end{equation} 
\end{itemize}
By the way the funcction $q \rightarrow \tan (hq)$ is a periodic function with period $\pi/h$. Taking this account with \eqref{eq-1-11} and  \eqref{eq-1-12}, we find that
\begin{equation*}
\alpha_1 \in \left(0, \frac{\pi}{2h} \right)
\end{equation*}
and for $m \geq 2$, 
\begin{equation*}
\alpha_m \in \left( \frac{(m-1)\pi}{h}, \frac{(m-1)\pi}{h} + \frac{\pi}{2h}\right).
\end{equation*}
The proof is done.
\end{proof}
For the first solution $\alpha_1$ to \eqref{eq-1-19}, we have the following result. 
\begin{lem}\label{lem-3-3} Assume that $ha \leq 1$. Then we have  
\begin{equation}\label{eq-1-20}
\alpha_1 \leq \sqrt{a^2 + \frac{2a}{h}} \leq \frac{\sqrt{3a}}{\sqrt{h}}.
\end{equation}
In addition, if we further assume that $ha \leq 1/3$, then we have
\begin{equation*}
\sqrt{a^2 + \frac{2a}{h +2ah^2}} \leq \alpha_1
\end{equation*}
From the above estimates, we find that $\lim_{h \rightarrow 0} \frac{\alpha_1}{\sqrt{2a/h}} =1.$
\end{lem}
\begin{proof}
Recall from Lemma \ref{lem-2-1} that $\alpha_1 >0$ is the smallest positive solution $q>0$ to 
\begin{equation}\label{eq-1-87}
\tan (hq) = \frac{2aq}{q^2 -a^2}.
\end{equation}
By an elementary calculus, the function $z \rightarrow \tan z$  has the following estimate
\begin{equation}\label{eq-r-1}
z + \frac{z^3}{3} \leq \tan z \leq z + \frac{2z^3}{3}\quad \textrm{for}~ z\in[0,1].
\end{equation}
We see that $z \rightarrow \tan (hz)$ is increasing for $z \in (0,\pi/2h)$. Let us set $z_0 = {\sqrt{\frac{2a}{h} +a^2}}$. If $z_0 < \frac{\pi}{2h}$ we have
\begin{equation}
\tan (hz_0)  \geq hz_0 = \frac{2az_0}{z_0^2 -a^2}.
\end{equation}
 Combining this with \eqref{eq-1-11} and \eqref{eq-1-12} we deduce $\alpha_1 \leq z_0$. In the case $z_0 \geq \frac{\pi}{2h}$, we have $\alpha_1 \leq z_0$ by Lemma \ref{lem-3-2}. Hence the first inequality of \eqref{eq-1-20} holds true. This also implies $\alpha_1 \leq \frac{\sqrt{3a}}{\sqrt{h}}$ because we have $a^2 \leq \frac{a}{h}$ from the condition $ha \leq 1$.

Assume that $ha \leq 1/3$. Then we have $h\alpha_1 \leq \sqrt{3ah} \leq 1$. Combining this with the second inequality of \eqref{eq-r-1}, we deduce
\begin{equation*}
\begin{split}
 \frac{2a\alpha_1}{\alpha_1^2 -a^2} = \tan (h\alpha_1) &\leq h\alpha_1+ \frac{2(h\alpha_1)^3}{3} \leq h\alpha_1 + 2h^2 a \alpha_1,
\end{split}
\end{equation*}
where we used $h \alpha_1 \leq \sqrt{3ah}$ in the second inequality. Rearranging this, we get
\begin{equation*}
a^2 + \frac{2a}{h+ 2h^2 a} \leq \alpha_1^2.
\end{equation*}
This completes the proof of this lemma.
\end{proof}

\begin{lem}For  $0 \leq z, w \leq h$ with $h \leq \frac{1}{3a}$ we have
\begin{equation}\label{eq-1-86}
\left| P_1 (z,t,w,s) - \frac{\alpha_1^2}{2a} e^{-\alpha_1^2 (t-s)} \right| \leq \frac{5}{2} \,h\alpha_1^2 e^{-\alpha_1^2 (t-s)}.
\end{equation}
\end{lem}
\begin{proof}

We recall from \eqref{eq-1-17} that $P_1 (z, t,w, s)$ is given by 
\begin{equation*}
P_1 (z,t,w,s)=\frac{2e^{-\alpha_1^2 (t-s)}[ \alpha_1 \cos (\alpha_1 z) + a \sin (\alpha_1 z)][\alpha_1 \cos (\alpha_1 w) + a \sin (\alpha_1 w)]}{2a + h (a^2 + \alpha_1^2)}.
\end{equation*}
Throughout the proof, we keep in mind that $\alpha_1^2 z^2 \leq \alpha_1^2 h^2 \leq 3ha \leq 1$ from Lemma \ref{lem-3-3}.
Since $1-\frac{v^2}{{2}} \leq \cos v \leq 1$ for $v \in \mathbb{R}$, we have 
\begin{equation*}
1 - \frac{\alpha_1^2 z^2}{2} \leq \cos (\alpha_1 z) \leq 1.
\end{equation*}
It then follows using $\alpha_1^2 z^2 \leq 3ha$ that 
\begin{equation*}
\alpha_1 \left(1-\frac{3ha}{2}\right) \leq \alpha_1 \cos (\alpha_1 z ) \leq \alpha_1,
\end{equation*}
which gives
\begin{equation*}
\alpha_1 \left(1-\frac{3ha}{2}\right) \leq \alpha_1 \cos (\alpha_1 z) + a \sin (\alpha_1 z) \leq \alpha_1 (1+ha),
\end{equation*}
where we used $0 \leq \alpha_1 z \leq h \alpha_1 \leq 1$. From this we obtain
\begin{equation}\label{eq-1-82} 
\alpha_1^2 \left(1-\frac{3ha}{2}\right)^2 \leq [\alpha_1 \cos (\alpha_1 z) +a \sin (\alpha_1 z)][\alpha_1 \cos (\alpha_1 w) + a \sin (\alpha_1 w)] \leq \alpha_1^2 (1+ha)^2.
\end{equation}
From Lemma \ref{lem-3-3} we find 
\begin{equation}\label{eq-1-61}
\frac{2a}{1 + 2ah} + a^2 h \leq h \alpha_1^2 \leq {2a} + a^2 h.
\end{equation}
Combining this with \eqref{eq-1-61} we find that $D:= 2a + h(a^2 +\alpha_1^2)$ satisfies
\begin{equation}\label{eq-1-62}
D \geq \left( \frac{2a}{1+2ah} + a^2 h \right) + 2a + ha^2 \geq 4a - 2a^2 h
\end{equation}
and
\begin{equation}\label{eq-1-63}
D \leq \left( {2a} + a^2 h \right) + 2a +ha^2 \leq 4a + 2a^2 h.
\end{equation}
Combining \eqref{eq-1-82} with \eqref{eq-1-62} and \eqref{eq-1-63} we deduce
\begin{equation}\label{eq-1-64}
2e^{-\alpha_1^2 (t-s)} \frac{\alpha_1^2 (1-2ha)^2}{4a + 2a^2 h} \leq P_1 (z,t,w,s) \leq 2e^{-\alpha_1^2 (t-s)} \frac{\alpha_1^2 (1+3ah)}{4a-2a^2 h}.
\end{equation}
Using that $a h \leq 1/3$ we have
\begin{equation*}
\frac{(1-2ha)^2}{4a + 2a^2h} \geq (1-4ha) \frac{1}{4a(1+ah/2)} \geq \frac{1}{4a} (1-4ha)(1-ah/2) \geq \frac{1}{4a} - \frac{9h}{16}.
\end{equation*}
Similarly,
\begin{equation*}
\frac{1+3ah}{4a-2a^2 h} = \frac{1+3ah}{4a (1-ah/2)} \leq \frac{1}{4a} \left[ (1+3ah) (1+ah)\right] \leq \frac{1}{4a} \left( 1+ 5ah\right) = \frac{1}{4a} + \frac{5h}{4}.
\end{equation*}
Gathering the above two estimates in \eqref{eq-1-64}, we obtain
\begin{equation*}
2e^{-\alpha_1^2 (t-s)} \alpha_1^2 \left( \frac{1}{4a} - \frac{9h}{16}\right) \leq P_1 (z,t,w,s) \leq 2e^{-\alpha_1^2 (t-s)}\alpha_1^2 \left( \frac{1}{4a} + \frac{5h}{4}\right).
\end{equation*}
From this we find
\begin{equation*}
\left| P_1 (z,t,w,s) - \frac{\alpha_1^2}{2a} e^{-\alpha_1^2 (t-s)} \right| \leq \frac{5}{2}\alpha_1^2 e^{-\alpha_1^2 (t-s)} h.
\end{equation*}
The proof is finished.
\end{proof}
\begin{lem}\label{lem-1-2}
Assume that $ah \leq 1/3$. Then we have
\begin{equation*}
\begin{split}
&\left| \int_0^t \int_{P \times \{0,h\}} W (x,t,y,s)P_1 (x_3,t,y_3,s) F(y,y_3,s)dS_y ds\right.
\\
& -\left.\int_0^{t} \int_{P \times \{0,h\}} 2e^{-\alpha_1^2 (t-s)} \frac{\alpha_1^2}{4a}W(x,t,y,s) F(y,y_3, s) dS_y ds \right| \leq 5h \|F\|_{L^{\infty}( [0,t])}.
\end{split}
\end{equation*}
\end{lem}
\begin{proof}
Using \eqref{eq-1-86} and \eqref{eq-1-76} we deduce
\begin{equation*}
\begin{split}
&\left| \int_0^t \int_{P \times \{0,h\}} W(x,t,y,s) P_1 (x_3, y, y_3, s) F(y,y_3, s) dS_y ds\right.
\\
&\quad  - \left.\int_0^t \int_{P \times \{0,h\}} \frac{\alpha_1^2}{2a} e^{-\alpha_1^2 (t-s)} W(x,t,y,s) F(y,y_3, s) dS_y ds \right|
\\
& \quad \leq \frac{5h}{2} \left| \int_0^t \int_{P \times \{0,h\}} \alpha_1^2 e^{-\alpha_1^2 (t-s)} W(x,t,y,s) F(y,y_3, s) dS_y ds \right|
\\
&\quad \leq 5h \|F\|_{L^{\infty}( [0,t])} \left( \int_0^{\infty} e^{-s} ds \right)
\\
& \quad = 5h \|F\|_{L^{\infty}( [0,t])}.
\end{split}
\end{equation*}
The proof is finished.
\end{proof}

Now we give the proof of Theorem \ref{thm-1}.
\begin{proof}
From \eqref{eq-1-2}, \eqref{eq-1-3}, and \eqref{eq-1-4} we have
\begin{equation}\label{eq-1-22}
\begin{split}
&U((x,x_3),t)
\\
& = G(x,x_3) +  \sum_{m=1}^{\infty} \int_0^t \int_{P \times \{0,h\}]} W(x,t,y,s) P_m (x_3,t,y_3,s) [F(y,y_3,s)] dS_y ds.
\end{split}
\end{equation}
Using Lemma \ref{lem-1-1} we deduce 
\begin{equation}\label{eq-1-70}
\begin{split}
&  \sum_{m=2}^{\infty} \int_0^t \int_{P \times \{0,h\}]} W(x,t,y,s) P_m (x_3,t,y_3,s) [F(y,y_3,s)] dS_y ds
\\
&\leq \sum_{m=2}^{\infty} \frac{8}{h \alpha_m^2}\|F\|_{L^{\infty}( [0,t])},
\end{split}
\end{equation}
and the estimate of Lemma \ref{lem-3-2} enables us to obtain
\begin{equation*}
\sum_{m=2}^{\infty} \frac{1}{h\alpha_m^2} \leq \sum_{m=2}^{\infty} \frac{h}{(m-1)^2 \pi^2} = \frac{h}{6}.
\end{equation*}
Gathering this together with \eqref{eq-1-70} and Lemma \ref{lem-1-2}, we finally deduce from \eqref{eq-1-22} the following estimate 
\begin{equation*}
\begin{split}
\left|U((x,x_3),t) - \left( G(x,x_3) +  \int_0^{t} \int_{P \times \{0\}} 2e^{-\alpha_1^2 (t-s)} \frac{\alpha_1^2}{4a}W(x,t,y,s) F(y,y_3, s) dS_y ds\right)\right| 
\\
\leq \left( 5+ \frac{4}{3}\right) h \|F\|_{L^{\infty}( [0,t])}.
\end{split}
\end{equation*}
The proof is finished.
\end{proof}


\end{document}